\documentclass[12pt,reqno]{article}

\usepackage[square, comma, sort&compress, numbers]{natbib}
\usepackage{amsmath,amsthm,amsfonts,amssymb,latexsym,mathrsfs,color}

\newtheorem{thm}{Theorem}[section]
\newtheorem{lem}[thm]{Lemma}
\newtheorem{coro}[thm]{Corollary}

\theoremstyle{definition}

\newtheorem{exm}[thm]{Example}
\newtheorem{rem}[thm]{Remark}

\numberwithin{equation}{section}

\title{Catalan-like numbers and Stieltjes moment sequences}
\author{Huyile Liang, Lili Mu, Yi Wang
\thanks{{\it Email address:}\quad wangyi@dlut.edu.cn (Y. Wang)}}
\date{\footnotesize School of Mathematical Sciences, Dalian University of Technology, Dalian 116024, PR China}

\begin{document}

\maketitle

\begin{abstract}
We provide sufficient conditions under which the Catalan-like numbers are Stieltjes moment sequences.
As applications, we show that many well-known counting coefficients,
including the Bell numbers, the Catalan numbers,
the central binomial coefficients, the central Delannoy numbers,
the factorial numbers, the large and little Schr\"oder numbers,
are Stieltjes moment sequences in a unified approach.
\bigskip
\\[1pt]
\end{abstract}

\section{Introduction}
\hspace*{\parindent}

A sequence $(m_n)_{n\ge 0}$ of numbers is said to be a {\it Stieltjes moment sequence} if it has the form
\begin{equation}\label{i-e}
m_n=\int_0^{+\infty}x^nd\mu(x),
\end{equation}
where $\mu$ is a non-negative measure on $[0,+\infty)$.
It is well known that $(m_n)_{n\ge 0}$ is a Stieltjes moment sequence if and only if
$\det[m_{i+j}]_{0\le i,j\le n}\ge 0$ and $\det[m_{i+j+1}]_{0\le i,j\le n}\ge 0$ for all $n\ge 0$
\cite[Theorem 1.3]{ST43}.
Another characterization for Stieltjes moment sequences comes from the theory of total positivity.

Let $A=[a_{n,k}]_{n,k\ge 0}$ be a finite or an infinite matrix.
It is {\it totally positive} ({\it TP} for short),
if its minors of all orders are nonnegative.
Let $\alpha=(a_n)_{n\ge 0}$ be an infinite sequence of nonnegative numbers.
Define the {\it Hankel matrix} $H(\alpha)$ of the sequence $\alpha$ by
$$H(\alpha)=[a_{i+j}]_{i,j\ge 0}
=\left[\begin{array}{lllll}
a_{0} & a_1 & a_2 & a_3 & \cdots\\
a_{1} & a_{2} & a_3 & a_4 & \cdots\\
a_{2} & a_{3} & a_{4} & a_5 & \cdots\\
a_{3} & a_{4} & a_{5} & a_{6} & \cdots\\
\vdots & \vdots &\vdots & \vdots & \ddots\\
\end{array}\right].$$
Then $\alpha$ is a Stieltjes moment sequence
if and only if $H(\alpha)$ is totally positive (see \cite[Theorem 4.4]{Pin10} for instance).

Many counting coefficients are Stieltjes moment sequences.
For example, the factorial numbers $n!$ form a Stieltjes moment sequence since
\begin{equation}\label{fn-i}
n!=\int_0^{\infty}x^ne^{-x}dx.
\end{equation}
The Bell numbers $B_n$ form a Stieltjes moment sequence
since $B_n$ can be interpreted as the $n$th moment of a Poisson distribution with expected value $1$
by Dobinski's formula
$$B_n=\frac{1}{e}\sum_{k\ge 0}\frac{k^n}{k!}.$$ 
The Catalan numbers $C_n=\binom{2n}{n}/(n+1)$
form a Stieltjes moment sequence since
$$\det[C_{i+j}]_{0\le i,j\le n}=\det[C_{i+j+1}]_{0\le i,j\le n}=1,\qquad n=0,1,2,\ldots$$
(see Aigner \cite{Aig99} for instance).
Bennett \cite{Ben11} showed that the central Delannoy numbers $D_n$
and the little Schr\"oder numbers $S_n$ form Stieltjes moment sequences by means of their generating functions
(see Remark \ref{ben} and Example \ref{ex2}).
All these counting coefficients are the so-called Catalan-like numbers.
In this note we provide sufficient conditions such that the Catalan-like numbers are Stieltjes moment sequences
by the total positivity of the associated Hankel matrices.
As applications,
we show that the Bell numbers, 
the Catalan numbers,
the central binomial coefficients,
the central Delannoy numbers,
the factorial numbers,
the large and little Schr\"oder numbers
are Stieltjes moment sequences in a unified approach.

\section{Main results and applications}
\hspace*{\parindent}

Let $\sigma=(s_k)_{k\ge 0}$ and $\tau=(t_k)_{k\ge 1}$ be two sequences of nonnegative numbers
and define an infinite lower triangular matrix $R:=R^{\sigma,\tau}=[r_{n,k}]_{n,k\ge 0}$
by the recurrence
\begin{equation}\label{rst-eq}
r_{0,0}=1,\qquad r_{n+1,k}=r_{n,k-1}+s_kr_{n,k}+t_{k+1}r_{n,k+1},
\end{equation}
where $r_{n,k}=0$ unless $n\ge k\ge 0$.
Following Aigner~\cite{Aig01},
we say that $R^{\sigma,\tau}$ is the {\it recursive matrix}
and $r_{n,0}$ are the {\it Catalan-like numbers}
corresponding to $(\sigma,\tau)$.

The Catalan-like numbers unify many well-known counting coefficients, such as
\begin{itemize}
  \item [\rm (1)] the Catalan numbers $C_n$ if $\sigma=(1,2,2,\ldots)$ and $\tau=(1,1,1,\ldots)$;
  \item [\rm (2)] the central binomial coefficients $\binom{2n}{n}$ if $\sigma=(2,2,2,\ldots)$ and $\tau=(2,1,1,\ldots)$;
  \item [\rm (3)] the central Delannoy numbers $D_n$ if $\sigma=(3,3,3,\ldots)$ and $\tau=(4,2,2,\ldots)$;
  \item [\rm (4)] the large Schr\"oder numbers $r_n$ if $\sigma=(2,3,3,\ldots)$ and $\tau=(2,2,2,\ldots)$;
  \item [\rm (5)] the little Schr\"oder numbers $S_n$ if $\sigma=(1,3,3,\ldots)$ and $\tau=(2,2,2,\ldots)$;
  \item [\rm (6)] the (restricted) hexagonal numbers $h_n$ if $\sigma=(3,3,3,\ldots)$ and $\tau=(1,1,1,\ldots)$;
  \item [\rm (7)] the Bell numbers $B_n$ if $\sigma=\tau=(1,2,3,4,\ldots)$;
  \item [\rm (8)] the factorial numbers $n!$ if $\sigma=(1,3,5,7,\ldots)$ and $\tau=(1,4,9,16,\ldots)$.
\end{itemize}

Rewrite the recursive relation \eqref{rst-eq} as
$$
\left[
\begin{array}{ccccc}
r_{1,0} & r_{1,1} &  &  &  \\
r_{2,0} & r_{2,1} & r_{2,2} &  &  \\
r_{3,0} & r_{3,1} & r_{3,2} & r_{3,3} & \\
& & \cdots  & & \ddots   \\
\end{array}
\right]
=
\left[
\begin{array}{cccc}
r_{0,0} &  &  &  \\
r_{1,0} & r_{1,1} &  &  \\
r_{2,0} & r_{2,1} & r_{2,2} &  \\
& \cdots &  & \ddots   \\
\end{array}
\right]
\left[
\begin{array}{cccc}
s_0 & 1 &  &\\
t_1 & s_1 & 1 &\\
 & t_2 & s_2 & \ddots\\
& &\ddots & \ddots \\
\end{array}\right],
$$
or briefly,
\begin{equation*}\label{AJ}
\overline{R}=RJ
\end{equation*}
where $\overline{R}$ is obtained from $R$ by deleting the $0$th row and $J$ is the tridiagonal matrix
\begin{equation*}\label{J-eq}
J:=J^{\sigma,\tau}=\left[
\begin{array}{ccccc}
s_0 & 1 &  &  &\\
t_1 & s_1 & 1 &\\
 & t_2 & s_2 & 1 &\\
 & & t_3 & s_3 & \ddots\\
& & &\ddots & \ddots \\
\end{array}\right].
\end{equation*}
Clearly, the recursive relation \eqref{rst-eq} is decided completely by the tridiagonal matrix $J$.
Call $J$ the {\it coefficient matrix} of the recursive relation \eqref{rst-eq}.

\begin{thm}\label{jh-lem}
If the coefficient matrix is totally positive,
then the corresponding Catalan-like numbers form a Stieltjes moment sequence.
\end{thm}
\begin{proof}
Let $H=[r_{n+k,0}]_{n,k\ge0}$ be the Hankel matrix of the Catalan-like numbers $(r_{n,0})_{n\ge 0}$.
We need to show that $H$ is totally positive.
We do this by two steps.
We first show the total positivity of the coefficient matrix $J$ implies that of the recursive matrix $R$.
Let $R_n=[r_{i,j}]_{0\le i,j\le n}$ be the $n$th leading principal submatrix of $R$.
Clearly, to show that $R$ is TP, it suffices to show that $R_n$ are TP for $n\ge 0$.
We do this by induction on $n$.
Obviously, $R_0$ is TP.
Assume that $R_n$ is TP.
Then by \eqref{rst-eq}, we have
$$R_{n+1}=\left[\begin{array}{cc}1 & O\\ O & R_n\\\end{array}\right]L_n,$$
where
$$L_n=\left[
      \begin{array}{ccccccc}
        1 & &  &  &  &  &  \\
        s_0 & 1 &  &  &  &  &  \\
        t_1 & s_1 & 1 &  &  &  &  \\
         & t_2 & s_2 & \ddots &  &  &  \\
         &  & \ddots & \ddots & 1 &  &  \\
         &  &  & t_{n-1} & s_{n-1} & 1 &  \\
         &  &  &  & t_n & s_n & 1 \\
      \end{array}
    \right].$$
Clearly, the total positivity of $R_n$ implies that of $\left[\begin{array}{cc}1 & O\\ O & R_n\\\end{array}\right]$.
On the other hand, $J$ is TP, so is its submatrix $J_n$, as well as the matrix $L_n$.
Thus the product matrix $R_{n+1}$ is TP,
and $R$ is therefore TP by induction.


Secondly we show the total positivity of $R$ implies that of $H$.
Let $T_0=1, T_k=t_1\cdots t_k$ and 
$$T=\left[
      \begin{array}{cccc}
        T_0 &  &  &  \\
         & T_1 &  &  \\
         &  & T_2 &  \\
         &  &  & \ddots \\
      \end{array}
    \right].$$
Then it is not difficult to verify that $H=RTR^t$ (see \cite[(2.5)]{Aig01}).
Now $R$ is TP, and so is its transpose $R^t$. Clearly, $T$ is TP.
Thus the product $H$ is also TP.
This completes the proof.
\end{proof}

We now turn to the total positivity of tridiagonal matrices.
Such a matrix is often called a {\it Jacobi matrix}.
There are many well-known results about the total positivity of tridiagonal matrices.
For example,
a finite nonnegative tridiagonal matrix is totally positive if and only if
all its principal minors containing consecutive rows and columns are nonnegative
\cite[Theorem 4.3]{Pin10};
and in particular, an irreducible nonnegative tridiagonal matrix is totally positive
if and only if all its leading principal minors are positive \cite[Example 2.2, p. 149]{Min88}.
Clearly, $J^{\sigma,\tau}$ is irreducible.
So, to show the total positivity of $J^{\sigma,\tau}$,
it suffices to show that all its leading principal minors are positive.

\begin{exm}
For the Catalan-like numbers $n!$,
we have $s_k=2k+1$ and $t_k=k^2$.
It is not difficult to show that the $n$th leading principal minor of $J^{\sigma,\tau}$ is equal to $n!$.
Thus $J^{\sigma,\tau}$ is totally positive, and so the factorial numbers $n!$ form a Stieltjes moment sequence,
a well-known result.
\end{exm}

\begin{lem}\label{jtp-lem}
If $s_0\ge 1$ and $s_k\ge t_k+1$ for $k\ge 1$,
then the tridiagonal matrix $J^{\sigma,\tau}$ is totally positive.
\end{lem}
\begin{proof}
Let $D_n$ be the $n$th leading principal minor of $J^{\sigma,\tau}$.
It suffices to show that all $D_n$ are nonnegative.
We do this by showing the following stronger result:
$$1\le D_0\le D_1\le D_{n-2}\le D_{n-1}\le D_n\le\cdots.$$
Obviously, $D_0=s_0\ge 1$ and $D_1=s_0s_1-t_1\ge s_0=D_0$ since $s_1\ge t_1+1$.
Assume now that $D_{n-1}\ge D_{n-2}\ge 1$ for $n\ge 2$.
Then by expanding the determinant $D_n$ along the last row or column, we obtain
$$D_n=s_nD_{n-1}-t_nD_{n-2}\ge (s_n-t_n)D_{n-1}\ge D_{n-1}\ge 1$$
by $s_n\ge t_n+1$ and the induction hypothesis, as required.
The proof is complete.
\end{proof}


Combining Theorem \ref{jh-lem} and Lemma \ref{jtp-lem}, we obtain the following.

\begin{coro}\label{thm1}
If $s_0\ge 1$ and $s_k\ge t_k+1$ for $k\ge 1$,
then the Catalan-like numbers corresponding to $(\sigma,\tau)$ form a Stieltjes moment sequence.
\end{coro}

\begin{exm}
The Bell numbers $B_n$, the Catalan numbers $C_n$, the central binomial coefficients $\binom{2n}{n}$,
the (restricted) hexagonal numbers $H_n$, and the large Schr\"oder numbers $r_n$
form a Stieltjes moment sequence respectively.
\end{exm}

In what follows we apply Theorem \ref{jh-lem}
to the recursive matrix $R(p,q;s,t)=[r_{n,k}]_{n,k\ge 0}$ defined by
\begin{equation}\label{rr}
\left\{
  \begin{array}{ll}
    r_{0,0}=1,\quad r_{n+1,0}=pr_{n,0}+qr_{n,1},\\
    r_{n+1,k+1}=r_{n,k}+sr_{n,k+1}+tr_{n,k+2}.
  \end{array}
\right.
\end{equation}
The coefficient matrix of \eqref{rr} is
\begin{equation}\label{J-pqst}
J(p,q;s,t)=\left[
\begin{array}{ccccc}
p & 1 &  &  &\\
q & s & 1 &\\
 & t & s & 1 &\\
 & & t & s & \ddots\\
& & &\ddots & \ddots \\
\end{array}\right].
\end{equation}

The following result is a special case of \cite[Proposition 2.5]{CLW-EuJC}.

\begin{lem}\label{jpqst-lem}
The Jacobi matrix $J(p,q;s,t)$ is totally positive if and only if
$s^2\ge 4t$ and $p(s+\sqrt{s^2-4t})/2\ge q$.
\end{lem}

On the other hand,
$R(p,q;s,t)$ is also a Riordan array.
A {\it Riordan array}, denoted by $(d(x),h(x))$, is an infinite lower triangular matrix
whose generating function of the $k$th column is $x^kh^k(x)d(x)$ for $k=0,1,2,\ldots$,
where $d(0)=1$ and $h(0)\neq 0$ \cite{SGWW91}.
A Riordan array $R=[r_{n,k}]_{n,k\ge 0}$
can be characterized by two sequences
$(a_n)_{n\ge 0}$ and $(z_n)_{n\ge 0}$ such that
\begin{equation}\label{rrr-c}
r_{0,0}=1,\quad r_{n+1,0}=\sum_{j\ge 0}z_jr_{n,j},\quad r_{n+1,k+1}=\sum_{j\ge 0}a_jr_{n,k+j}
\end{equation}
for $n,k\ge 0$ (see \cite{HS09} for instance).
Let $Z(x)=\sum_{n\ge 0}z_nx^n$ and $A(x)=\sum_{n\ge 0}a_nx^n$.
Then it follows from \eqref{rrr-c} that
\begin{equation}\label{fg}
  d(x)=\frac{1}{1-xZ(xh(x))},\quad h(x)=A(xh(x)).
\end{equation}
Now $R(p,q;s,t)$ is a Riordan array with $Z(x)=p+qx$ and $A(x)=1+sx+tx^2$.
Let $R(p,q;s,t)=(d(x),h(x))$.
Then by \eqref{fg}, we have
$$d(x)=\frac{1}{1-x(p+qxh(x))},\quad h(x)=1+sxh(x)+tx^2h^2(x).$$
It follows that
$$h(x)=\frac{1-sx-\sqrt{1-2sx+(s^2-4t)x^2}}{2tx^2}$$
and
$$d(x)=\frac{2t}{2t-q+(qs-2pt)x+q\sqrt{1-2sx+(s^2-4t)x^2}}$$
(see \cite{WZ-LAA} for details).

Combining Theorem \ref{jh-lem} and Lemma \ref{jpqst-lem}, we have the following.

\begin{thm}\label{d-hp}
Let $p,q,s,t$ be all nonnegative and
$$\sum_{n\ge 0}d_nx^n=\frac{2t}{2t-q+(qs-2pt)x+q\sqrt{1-2sx+(s^2-4t)x^2}}.$$
If $s^2\ge 4t$ and $p(s+\sqrt{s^2-4t})/2\ge q$,
then $(d_n)_{n\ge 0}$ is a Stieltjes moment sequence.
\end{thm}

Setting $q=t$ in Theorem~\ref{d-hp}, we obtain

\begin{coro}\label{t=q}
Let $p,s,t$ be all nonnegative and
$$\sum_{n\ge 0}d_nx^n=\frac{2}{1+(s-2p)x+\sqrt{1-2sx+(s^2-4t)x^2}}.$$
If $s^2\ge 4t$ and $p(s+\sqrt{s^2-4t})/2\ge t$,
then $(d_n)_{n\ge 0}$ is a Stieltjes moment sequence.
\end{coro}

In particular, taking $p=s$ in Corollary \ref{t=q},
and noting $s^2\ge 4t$ implies that $s(s+\sqrt{s^2-4t})/2\ge t$, we have

\begin{coro}\label{t=q,s=p}
Let $s,t$ be nonnegative and
$$\sum_{n\ge 0}d_nx^n=\frac{2}{1-sx+\sqrt{1-2sx+(s^2-4t)x^2}}.$$
If $s^2\ge 4t$,
then $(d_n)_{n\ge 0}$ is a Stieltjes moment sequence.
\end{coro}

On the other hand, taking $p=s$ and $q=2t$ in Theorem~\ref{d-hp}, we obtain

\begin{coro}\label{2t=q,s=p}
Let $s,t$ be nonnegative and
$$\sum_{n\ge 0}d_nx^n=\frac{1}{\sqrt{1-2sx+(s^2-4t)x^2}}.$$
If $s^2\ge 4t$,
then $(d_n)_{n\ge 0}$ is a Stieltjes moment sequence.
\end{coro}

\begin{rem}\label{ben}
Corollaries \ref{t=q,s=p} and \ref{2t=q,s=p} have occurred in Bennett \cite[\S 9]{Ben11}.
\end{rem}

\begin{exm}\label{ex2}
The Catalan numbers $C_n$,
the central binomial coefficients $\binom{2n}{n}$,
the central Delannoy numbers $D_n$,
the large Schr\"oder numbers $r_n$,
the little Schr\"oder numbers $S_n$
have generating functions
\begin{gather*}
  \sum_{n\ge 0}C_nx^n=\frac{2}{1+\sqrt{1-4x}}, \\
  \sum_{n\ge 0}\binom{2n}{n}x^n=\frac{1}{\sqrt{1-4x}}, \\
  \sum_{n\ge 0}D_nx^n=\frac{1}{\sqrt{1-6x+x^2}}, \\
  \sum_{n\ge 0}r_nx^n=\frac{2}{1-x+\sqrt{1-6x+x^2}},\\
  \sum_{n\ge 0}S_nx^n=\frac{2}{1+x+\sqrt{1-6x+x^2}}.
\end{gather*}
respectively.
Again we see that these numbers are all Stieltjes moment sequences.
\end{exm}

\section{Remarks}
\hspace*{\parindent}

A Stieltjes moment sequence $(m_n)$ is called {\it determinate},
if there is a unique measure $\mu$ on $[0,+\infty)$ such that \eqref{i-e} holds;
otherwise it is called {\it indeterminate}.
For example, $(n!)$ is a Stieltjes moment sequence of the exponential distribution by \eqref{fn-i}
and determinate by Stirling's approximation and Carleman's criterion which states that the divergence of the series
$$\sum_{n\ge 0}\frac{1}{\sqrt[2n]{m_n}}$$
implies the determinacy of the moment sequence $(m_n)$ (see \cite[Theorem 1.11]{ST43} for instance).
We have shown that many well-known Catalan-like numbers are Stieltjes moment sequences.
However, we do not know how to obtain the associated measures in general
and whether these moment sequences are determinate.

\section*{Acknowledgement}

This work was supported in part by
the NSF of China
(Grant No. 11371078).

\end{document}